\documentclass[12pt]{article}

\usepackage{latexsym, amssymb, amsmath, amscd, amsfonts, epsfig, graphicx, colordvi,verbatim,ifpdf}
\usepackage{amsfonts, amsmath, amssymb}
\usepackage{amssymb,amsfonts,amsmath,latexsym,epsfig,cite, psfrag,eepic,color}
\usepackage{amscd,graphics}
\usepackage{latexsym, amssymb,  amsmath,amscd, amsfonts, epsfig, graphicx, colordvi,amsthm}
\usepackage{color}

\usepackage{graphicx}
\usepackage{color}
\usepackage{ifpdf}
\usepackage{fancybox}
\usepackage[colorlinks,urlcolor=purple,linkcolor=red,anchorcolor=green,citecolor=blue]{hyperref}
\usepackage{mathrsfs}
\usepackage{url}

\def\qed{\nopagebreak\hfill{\rule{4pt}{7pt}}
\medbreak}

\def\p{{\overline{p}}}
\def\wmu{{\widehat{\mu}}}
\def\wnu{{\widehat{\nu}}}
\def\wvar{{\widehat{\varepsilon}}}
\def\wT{{\widehat{T}}}
\def\wR{{\widehat{R}}}

\def\wy{{\widehat{y}}}
\def\wsR{{\hat{\mathscr{L}}}}
\def\u{{\overline{u}}}

\newlength{\boxedparwidth}
  {\begin{center} \begin{tabular}{|@{\hspace{.315in}}c@{\hspace{.15in}}|}
                  \hline \\ \begin{minipage}[t]{\boxedparwidth}
                  \setlength{\parindent}{.25in}}%
  {\end{minipage} \\ \\ \hline \end{tabular} \end{center}}

\allowdisplaybreaks

\newtheorem{thm}{Theorem}[section]

\newtheorem{lem}[thm]{Lemma}
\newtheorem{cor}[thm]{Corollary}

\numberwithin{equation}{section}

\begin{document}

\baselineskip 20pt

\newcommand{\lr}[1]{\langle #1 \rangle}
\newcommand{\llr}[1]{\langle\hspace{-2.5pt}\langle #1
\rangle\hspace{-2.5pt}\rangle}

\begin{center}

 {\Large \bf Higher order log-concavity of the overpartition function and its consequences}
\end{center}
\vskip 0.2cm

\begin{center}
{Gargi Mukherjee}$^{1}$, {Helen W.J. Zhang}$^{2, 3}$ and
  {Ying Zhong}$^{2}$ \vskip 2mm
$^{1}$Institute for Algebra\\[2pt]
Johannes Kepler University, Altenberger Straße 69, A-4040 Linz, Austria\\[5pt]

$^{2}$School of Mathematics\\[2pt]
Hunan University, Changsha 410082, P.R. China\\[5pt]

$^3$Hunan Provincial Key Laboratory of \\ Intelligent Information Processing and Applied Mathematics,
\\[2pt] Changsha 410082, P.R. China\\[5pt]

Email: gargi.mukherjee@dk-compmath.jku.at, \quad helenzhang@hnu.edu.cn, \quad YingZhong@hnu.edu.cn
\end{center}

\vskip 6mm \noindent {\bf Abstract.}
Let $\p(n)$ denote the overpartition function. In this paper, we study the asymptotic higher order $\log$-concavity property of the overpatition function in a similar framework done by Hou and Zhang for the partition function. This will enable us to move on further in order to prove $\log$-concavity of overpartitions, explicitly by studying the asymptotic expansion of the quotient $\p(n-1)\p(n+1)/\p(n)^2$ upto a certain order so that one can finally ends up with the phenomena of $2$-$\log$-concavity and higher order Tur\'{a}n property of $\p(n)$ by following a sort of unified approach.

\vskip 0.3cm



\section{Introduction}
A partition of a positive integer $n$ is a non-increasing sequence of positive integers whose sum is $n$. Let $p(n)$ denote the number of partitions of $n$.
Recall that a sequence $\{a_n\}_{n\geq0}$ is called log-concave if
\[a_n^2-a_{n+1}a_{n-1}\geq0,~~n\geq 1.\]
Moreover, this sequence is said to be asymptotically $r$-log-concave if there exists $N$ such that
\begin{align}
\wsR\{a_n\}_{n\geq N}, \wsR^2\{a_n\}_{n\geq N}, \ldots, \wsR^r\{a_n\}_{n\geq N}
\end{align}
are all non-negative sequences, where
\[\wsR\{a_n\}_{n\geq 0}=\{a_{n+1}^2-a_na_{n+2}\}_{n\geq0}\quad\text{and}
\quad
\wsR^k\{a_n\}_{n\geq 0}=\wsR\left(\wsR^{k-1}\{a_n\}_{n\geq 0}\right).\]

Based on the Hardy-Ramanujan-Rademacher formula \cite{Andrews-1998, Hardy-1940, Hardy-Ramanujan-1918, Rademacher-1938} and the error estimation given by Lehmer \cite{Lehmer-1938, Lehmer-1939},
DeSalvo and Pak \cite{DeSalvo-Pak-2015} showed that the partition function $p(n)$ is log-concave for all $n>25$, conjectured by Chen \cite{Chentalk-2010}. Consequently, Chen, Wang and Xie proved the DeSalvo-Pak conjecture that states
\begin{thm}[Conjecture 1.3, \cite{ChenWangXie-2016}]\label{Chen1}
For $n \geq 45$, we have
\begin{equation*}
\dfrac{p(n-1)}{p(n)}\Bigl(1+\dfrac{\pi}{\sqrt{24}n^{3/2}}\Bigr)>\dfrac{p(n)}{p(n+1)}.
\end{equation*}
\end{thm}
Recently, Chen, Jia and Wang \cite{ChenJiaWang-2019} proceed further to show that $\{p(n)\}_{n \geq 95}$ satisfies the higher order Tur\'{a}n inequality and formulate a conjecture \cite[Conjecture 1.2]{ChenJiaWang-2019} in somewhat similar to Theorem \ref{Chen1}, settled by Larson and Wagner \cite[Theorem 1.2]{LarsonWagner-2019}.
What's more, Hou and Zhang \cite{Hou-Zhang-2019} proved the asymptotic $r$-log-concavity of $p(n)$ and as a consequence they showed $\{p(n)\}_{n \geq 221}$ is $2$-$\log$-concave, whereas an alternative approach through studying determinant of certain class of matrix can be found in \cite{JiaWang-2020}.

The overpartition function also reflects the similar log-behavior.
Recall an overpartition \cite{Corteel-Lovejoy-2004} of a nonnegative integer $n$ is a partition of $n$ where the first occurrence of each distinct part may be overlined.
Let $\p(n)$ denote the number of overpartitions of $n$.
Zukermann  \cite{Zuckerman-1939} gave a formula for the overpartition function, which is considered by Sills\cite{Sills-2010} as a Rademacher-type convergent series
\begin{align}\label{overlinep-asym}
\p(n)=\frac{1}{2\pi}\sum_{k=1\atop 2\nmid k}^\infty\sqrt{k}\sum_{h=0\atop (h,k)=1}^k
\frac{\omega(h,k)^2}{\omega(2h,k)}e^{-\frac{2\pi inh}{k}}\frac{\mathrm{d}}{\mathrm{d}n}
\left(\frac{\sinh\frac{\pi\sqrt{n}}{k}}{\sqrt{n}}\right),
\end{align}
where
\[\omega(h,k)=\exp\left(\pi i\sum_{r=1}^{k-1}\frac{r}{k}\left(\frac{hr}{k}
-\left\lfloor\frac{hr}{k}\right\rfloor-\frac{1}{2}\right)\right)\]
for positive integers $h$ and $k$.

Let $\wmu(n)=\pi\sqrt{n}$.
From this Rademacher-type series \eqref{overlinep-asym}, Engel \cite{Engel-2017} provided an error term for the overpartition function
\begin{align*}
\p(n)=\frac{1}{2\pi}\sum_{k=1\atop 2\nmid k}^N\sqrt{k}\sum_{h=0\atop (h,k)=1}^k
\frac{\omega(h,k)^2}{\omega(2h,k)}e^{-\frac{2\pi inh}{k}}\frac{\mathrm{d}}{\mathrm{d}n}
\left(\frac{\sinh\frac{\wmu(n)}{k}}{\sqrt{n}}\right)+R_2(n,N),
\end{align*}
where
\begin{align}\label{R_2(n,N)}
|R_2(n,N)|\leq \frac{N^{\frac{5}{2}}}{n\wmu(n)}
\sinh\left(\frac{\wmu(n)}{N}\right).
\end{align}
In particular, when $N=3$, we have
\begin{align}\label{overlinep-asym-1}
\p(n)=\frac{1}{8n}\left[\left(1+\frac{1}{\wmu(n)}\right)e^{-\wmu(n)}+\left(1-\frac{1}{\wmu(n)}\right)e^{\wmu(n)}\right]
+R_2(n,3),
\end{align}
where
\begin{align}\label{R_2(n,3)}
|R_2(n,3)|\leq \frac{3^{\frac{5}{2}}}{n\wmu(n)}
\sinh\left(\frac{\wmu(n)}{3}\right)\leq \frac{3^{\frac{5}{2}}e^{\frac{\wmu(n)}{3}}}{2n\wmu(n)}.
\end{align}
Similar to the work done in the world of partitions, Engel initiated the study of $\log$-concavity property of the overpartition function in his work \cite{Engel-2017}. The second author and Liu established Theorem \ref{Chen1} in context of overpartitions in \cite[Equation (1.6)]{LiuZhang-2021}. They also proved the higher order Tur\'{a}n property of $\p(n)$ for $n \geq 16$ (cf. see \cite[Theorem 1.2]{LiuZhang-2021}). Following the treatment done in \cite{JiaWang-2020}, the first author \cite[Theorem 1.7]{Mukherjee-2022} laid out a proof of $\{\p(n)\}_{n \geq 42}$ is $2$-$\log$-concave.

In this paper, our main goal is to prove the asymptotic r-log-concavity for the overpartition function, stated in Theorem \ref{p/n-r-log}.
In the proof of Theorem \ref{p/n-r-log}, we give a bound for $\p(n+1)/\p(n)$ and an asymptotic expression of $\p(n-1)\p(n+1)/\p(n)^2$ which are complicated in calculation. Consequently, we shall study the asymptotic growth of the quotient $\p(n-1)\p(n+1)/\p(n)^2$ up to $n^{-4}$, stated in Theorem \ref{thm1} as a specific example which presents computing process.
This in turn helps for a further study of certain quotients stated in Theorems \ref{thm2} and \ref{thm3}.
A host of inequalities for overpartition function, see Corollaries \ref{cor1}-\ref{cor6}, appear as a special case of the theorems, similar to the partition ones as discussed before.
The primary objective of this paper is to exploit the proof of Theorem \ref{p/n-r-log}, so that one can bring in all the proofs of Corollaries \ref{cor1}-\ref{cor6} under a unique structure, unlike the different array of structure of proofs for inequalities in context of the partition function.




\begin{thm}\label{p/n-r-log}
The sequence $\{\p(n)\}_{n\geq1}$ is asymptotically $r$-log-concave for any positive integer $r$.
\end{thm}

Define
\[\u_n:=\dfrac{\p(n-1)\p(n+1)}{\p(n)^2}.\]
\begin{thm}\label{thm1}
For all $n \geq 37$, we have
\begin{equation}\label{eqn1}
s_n-\dfrac{15}{n^4}<\u_n<s_n+\dfrac{20}{n^4},
\end{equation}
where
\begin{equation*}
s_n=1-\dfrac{\pi}{4n^{3/2}}+\dfrac{1}{n^2}-\dfrac{3}{4\pi n^{5/2}}-\dfrac{32+\pi^4}{32\pi^2}\dfrac{1}{n^3}-\Bigl(\dfrac{5}{4\pi^3}+\dfrac{21\pi}{64}\Bigr)\dfrac{1}{n^{7/2}}.
\end{equation*}
\end{thm}

\begin{cor}\cite[Theorem 1.2]{Engel-2017}\label{cor1}
$\{\p(n)\}_{n \geq 4}$ is $\log$-concave.
\end{cor}

\begin{cor}\cite[Equation (1.6)]{LiuZhang-2021}\label{cor2}
For $n \geq 2$,
\begin{equation}\label{eqn2}
\dfrac{\p(n-1)}{\p(n)}\Bigl(1+\dfrac{\pi}{4n^{3/2}}\Bigr)>\dfrac{\p(n)}{\p(n+1)}.
\end{equation}	
\end{cor}

\begin{thm}\label{thm2}
For all $n \geq 27$,
\begin{equation}\label{eqn3}
t_n-\dfrac{120}{n^{5/2}}< \dfrac{(1-\u_n)^2}{\u^2_n(1-\u_{n-1})(1-\u_{n+1})}<t_n+\dfrac{120}{n^{5/2}},
\end{equation}
where
\begin{equation*}
t_n=1+\dfrac{\pi}{2n^{3/2}}-\dfrac{7}{2n^2}.
\end{equation*}	
\end{thm}

\begin{cor}\cite[Theorem 1.7]{Mukherjee-2022}\label{cor3}
	$\{\p(n)\}_{n \geq 42}$ is $2$-$\log$-concave.
\end{cor}

\begin{cor}\label{cor4}
	For $n \geq 52$,
	\begin{equation}\label{eqn4}
\u^2_n (1-\u_{n-1})(1-\u_{n+1})	\Bigl(1+\dfrac{\pi}{2n^{3/2}}\Bigr)>(1-\u_n)^2.
	\end{equation}	
\end{cor}

\begin{thm}\label{thm3}
	For all $n \geq 2$,
	\begin{equation}\label{eqn5}
	v_n-\dfrac{120}{n^{5/2}}< \dfrac{4(1-\u_n)(1-\u_{n+1})}{(1-\u_n\u_{n+1})^2}<v_n+\dfrac{101}{n^{5/2}},
	\end{equation}
	where
	\begin{equation*}
	v_n=1+\dfrac{\pi}{4n^{3/2}}-\dfrac{25}{16n^2}.
	\end{equation*}		
\end{thm}

\begin{cor}\cite[Theorem 1.2]{LiuZhang-2021}\label{cor5}
	$\{\p(n)\}_{n \geq 16}$ satisfies higher order Tur\'{a}n inequality.
\end{cor}

\begin{cor}\label{cor6}
	For $n \geq 2$,
	\begin{equation}\label{eqn6}
(1-\u_n\u_{n+1})^2	\Bigl(1+\dfrac{\pi}{4n^{3/2}}\Bigr)>4(1-\u_n)(1-\u_{n+1}).
	\end{equation}	
\end{cor}

The paper is organized as follow. Proof of Theorem \ref{p/n-r-log} is given in Section \ref{sec1}, first we obtain an error estimation of $\p(n)$ in Subsection \ref{subsec1} and then computing the asymptotic expression of $\mathscr{R}^2\p(n)=\p(n)\p(n+2)/\p(n+1)^2$ by studying the bounds for the ratio $\p(n+1)/\p(n)$ in Subsection \ref{subsec2}.
Proof of Theorems \ref{thm1}, \ref{thm2}, \ref{thm3} and Corollaries \ref{cor1}-\ref{cor6} is given in Section \ref{sec2}.

\section{Proof of Theorem \ref{p/n-r-log}}\label{sec1}
In this section, we utilize the Rademacher-type convergent series and the error
estimation given by Engel to derive an estimation for $\p(n)$.
In view of \eqref{overlinep-asym-1}, $\p(n)$ can be written as
\begin{align}\label{p-T-R}
\p(n)=\wT(n)+\wR(n),
\end{align}
where
\begin{align} \label{wT(n)}
\wT(n)&=\frac{1}{8n}\left(1-\frac{1}{\wmu(n)}\right)e^{\wmu(n)},
\\[5pt] \label{wR(n)}
\wR(n)&=\frac{1}{8n}\left(1+\frac{1}{\wmu(n)}\right)e^{-\wmu(n)}+R_2(n,3).
\end{align}
\subsection{Error estimation of $\p(n)$}\label{subsec1}

To obtain the error estimation of $\p(n)$, we need to introduce the following lemma.
\begin{lem}\label{Hou-Zhang}
For any integer $m\geq1$, there exists a real number
\begin{equation*}
N_0(m) :=
\begin{cases}
1, &\quad \text{if}\ m=1,\\
2m \log m-m \log \log m, & \quad \text{if}\ m \geq 2,
\end{cases}
\end{equation*}
such that
\begin{align*}
x^me^{-x}<1,~~~\text{for}~~x\geq N_0(m).
\end{align*}
\end{lem}
\begin{proof}
	For $m=1$, it is immediate that $N_0(m)=1$. For $m \geq 2$, rewrite the inequality $x^me^{-x}<1$ as $f(x):=x-m\log x>0$. Now $f(x)$ is strictly increasing for $x>m$.  In order to show $f(x)>0$ for $x\geq N_0(m)$, first we show that $N_0(m)>m$ and then it is enough to show $f(N_0(m))>0$. To prove $N_0(m)>m$, it is equivalent to show $m^2>e. \log m$ which holds for $m \geq 2$. Next, we observe that
	\begin{equation}\label{lemeqn1}
	f(N_0(m))>0 \Leftrightarrow\log m  >  \log 2+2 \log \log m+\log \Bigl(1-\dfrac{\log \log m}{2\log m}\Bigr)\nonumber\\
	\end{equation}
For $m \geq 3$, $\log \Bigl(1-\dfrac{\log \log m}{2\log m}\Bigr) <0$ and hence, it is sufficient to prove $$\log m  >  \log 2+2 \log \log m \Leftrightarrow m>2\ (\log m)^2,$$
which holds for $m \geq 14$. Therefore, $f(N_0(m))>0$ for all $m \geq 14$ and we conclude the proof by checking numerically that $f(N_0(m))>0$ for $2\leq m \leq 13$.
\end{proof}

With the aid of Lemma \ref{Hou-Zhang}, we obtain the following conclusion.
\begin{thm}\label{wyn-bound}
For any integer $m\geq2$, there exist an integer $N_1(m)$ with
\begin{align*}
N_1(m)=\max\left\{184,\dfrac{9}{4\pi^2} N^2_0(m)\right\},
\end{align*}
such that
\begin{align*}
\left|\wy_n\right|<\left(\frac{3}{2}\right)^{m+1}\wmu(n)^{-m},
\end{align*}
where $\wy_n=\wR(n)/\wT(n)$.
\end{thm}

\proof
By \eqref{wT(n)} and \eqref{wR(n)}, we have
\begin{align}\label{wyn-T-R}
\wy_{n}\leq e^{-\frac{2\wmu(n)}{3}}\left(\wT_1(n)+\wR_1(n)\right),
\end{align}
where
\begin{align*}
\wT_1(n)=\frac{\wmu(n)+1}{\wmu(n)-1}e^{-\frac{4\wmu(n)}{3}},
~~~~~~
\wR_1(n)=4\cdot3^{\frac{5}{2}}\frac{1}{\wmu(n)-1}.
\end{align*}
From \eqref{wyn-T-R}, it follows that
\begin{align*}
\wT_1(184)+\wR_1(184)<\frac{3}{2}.
\end{align*}
Therefore,
\begin{align*}
\left|\wy_n\right|&<\frac{3}{2}e^{-\frac{2\wmu(n)}{3}},~~~\text{for}~~n\geq 184.
\end{align*}
According to Lemma \ref{Hou-Zhang}, there exists the integer $N_0(m)$ such that for $\frac{2}{3}\wmu(n)\geq N_0(m)$
\begin{align*}
e^{-\frac{2\wmu(n)}{3}}<\left(\frac{3}{2}\right)^m\wmu(n)^{-m}.
\end{align*}
On the other side,
\begin{align*}
n^{\frac{1}{2}}=\frac{\wmu(n)}{\pi}.
\end{align*}
Therefore, when
\begin{align*}
n\geq\max\left\{184, \dfrac{9}{4\pi^2} N^2_0(m)\right\}=N_1(m),
\end{align*}
we have
\begin{align*}
\wmu(n)>\dfrac{3}{2}N_0(m),
\end{align*}
which completes the proof.
\qed

\subsection{Bounds for the ratio $\p(n+1)/\p(n)$}\label{subsec2}
In order to obtain an estimation of $\p(n+1)/\p(n)$, we need the following lower and upper bounds for $\wT(n+1)/\wT(n)$.
\begin{lem}\label{T-T-bound}
Let $\wnu(n)=\wmu(n)(\wmu_1(n)-1)$ and
\begin{align*}
\wmu_1(n)&=\sum_{k=0}^{m'}{1/2\choose k}\pi^{2k}\wmu(n)^{-2k},\qquad
\wvar_1(n)=\left|{1/2\choose m'+1}\right|\pi^{2(m'+1)}\wmu(n)^{-2(m'+1)},
\\[3pt]
\wmu_2(n)&=\sum_{k=0}^{m'}{-3/2\choose k}\pi^{2k}\wmu(n)^{-2k}, \qquad
\wvar_2(n)=\left|{-3/2\choose m'+1}\right|\pi^{2(m'+1)}\wmu(n)^{-2(m'+1)},
\\[3pt]
\wnu_1(n)&=\left(\wmu_1(n)-\wvar_1(n)-\frac{1}{\wmu(n)}\right)\sum_{k=0}^m\wmu(n)^{-k},
\\[3pt]
\wnu_2(n)&=\left(\wmu_1(n)+\wvar_1(n)-\frac{1}{\wmu(n)}\right)
\left(\sum_{k=0}^m\wmu(n)^{-k}+2\wmu(n)^{-m-1}\right),
\end{align*}
then
\begin{align}\label{T-ratio-low}
\frac{\wT(n+1)}{\wT(n)}&>\wnu_1(n)\left(\wmu_2(n)-\wvar_2(n)\right)
(1-\wmu(n)\wvar_1(n))\sum_{k=0}^m\frac{\wnu(n)^k}{k!},
\end{align}
and
\begin{align}\label{T-ratio-upp}
\frac{\wT(n+1)}{\wT(n)}&<\wnu_2(n)(\wmu_2(n)+\wvar_2(n))
(1+2\wmu(n)\wvar_1(n))\left(\sum_{k=0}^m\frac{\wnu(n)^k}{k!}
+e^{\wnu(n)}\frac{\wnu(n)^{m+1}}{(m+1)!}\right)
\end{align}
where $\wT(n)$ is defined as in \eqref{wT(n)}.
\end{lem}

\proof
By \eqref{wT(n)}, we have
\begin{align}\label{T-ratio}
\frac{\wT(n+1)}{\wT(n)}=\frac{\wmu(n+1)-1}{\wmu(n)-1}\cdot\frac{\wmu(n)^3}{\wmu(n+1)^3}\cdot e^{\wmu(n+1)-\wmu(n)}.
\end{align}
Now we consider the above ratio term by term.
For the first factor, we have
\begin{align*}
\frac{\wmu(n+1)-1}{\wmu(n)-1}=\frac{\frac{\wmu(n+1)}{\wmu(n)}-\frac{1}{\wmu(n)}}{1-\frac{1}{\wmu(n)}}.
\end{align*}
By Taylor's Theorem, we have
\begin{align*}
\left(1-\frac{1}{\wmu(n)}\right)^{-1}
=\sum_{k=0}^\infty \wmu(n)^{-k},
\end{align*}
which implies that
\begin{align}\label{wT-n+1-1-1}
\sum_{k=0}^m\wmu(n)^{-k}<\left(1-\frac{1}{\wmu(n)}\right)^{-1}
<\sum_{k=0}^m\wmu(n)^{-k}+2\wmu(n)^{-m-1}.
\end{align}
Noting that
\begin{align*}
\wmu(n+1)=\wmu(n)\left(1+\frac{\pi^2}{\wmu(n)^2}\right)^{\frac{1}{2}}.
\end{align*}
For any integer $m$, let $m'=\lfloor\frac{m}{2}\rfloor$. Since
\begin{align*}
\left(1+\frac{\pi^2}{\wmu(n)^2}\right)^{\frac{1}{2}}=\sum_{k=0}^{m'}{1/2\choose k}\pi^{2k}\wmu(n)^{-2k}
+{1/2\choose m'+1}\left(\frac{\pi^2}{\wmu(n)^2}\right)^{m'+1}\left(1+\xi\right)^{\frac{1}{2}-m'-1},
\end{align*}
where $0<\xi<\frac{\pi^2}{\wmu(n)^2}$.
We have
\begin{align}\label{wT-n+1-1-2}
\wmu_1(n)-\wvar_1(n)<\frac{\wmu(n+1)}{\wmu(n)}<\wmu_1(n)+\wvar_1(n).
\end{align}
Combining \eqref{wT-n+1-1-1} and \eqref{wT-n+1-1-2}, we deduce that
\begin{align}\label{wT-n+1-1}
\wnu_1(n)<\frac{\wmu(n+1)-1}{\wmu(n)-1}<\wnu_2(n).
\end{align}
For the second factor, we have
\begin{align*}
\frac{\wmu(n)^3}{\wmu(n+1)^3}=\left(1+\frac{\pi^2}{\wmu(n)^2}\right)^{-3/2}.
\end{align*}
Since
\begin{align*}
\left(1+\frac{\pi^2}{\wmu(n)^2}\right)^{-\frac{3}{2}}=\sum_{k=0}^{m'}{-3/2\choose k}\pi^{2k}\wmu(n)^{-2k}
+{-3/2\choose m'+1}\left(\frac{\pi^2}{\wmu(n)^2}\right)^{m'+1}\left(1+\xi\right)^{-\frac{3}{2}-m'-1},
\end{align*}
where $0<\xi<\frac{\pi^2}{\wmu(n)^2}$.
We have
\begin{align}\label{wT-n+1-2}
\wmu_2(n)-\wvar_2(n)<\frac{\wmu(n)^3}{\wmu(n+1)^3}<\wmu_2(n)+\wvar_2(n).
\end{align}
For the last factor, using \eqref{wT-n+1-1-2}, then
\begin{align*}
e^{\wnu(n)-\wmu(n)\wvar_1(n)}<e^{\wmu(n+1)-\wmu(n)}<e^{\wnu(n)+\wmu(n)\wvar_1(n)}.
\end{align*}
Evidently, for $0<x<\frac{1}{2}$,
\[e^{-x}>1-x,\qquad e^x<1+2x\]
and for $x>0$
\[\sum_{k=0}^m\frac{x^k}{k!}<e^x<\sum_{k=0}^m\frac{x^k}{k!}+e^x\frac{x^{m+1}}{(m+1)!},\]
so that
\begin{align}\label{wT-n+1-3-low}
e^{\wmu(n+1)-\wmu(n)}>(1-\wmu(n)\wvar_1(n))\sum_{k=0}^m\frac{\wnu(n)^k}{k!}
\end{align}
and
\begin{align}\label{wT-n+1-3-upp}
e^{\wmu(n+1)-\wmu(n)}<(1+2\wmu(n)\wvar_1(n))\left(\sum_{k=0}^m\frac{\wnu(n)^k}{k!}
+e^{\wnu(n)}\frac{\wnu(n)^{m+1}}{(m+1)!}\right).
\end{align}
Applying the estimates \eqref{wT-n+1-1}, \eqref{wT-n+1-2}, \eqref{wT-n+1-3-low} and \eqref{wT-n+1-3-upp} to \eqref{T-ratio}, we reach \eqref{T-ratio-low} and \eqref{T-ratio-upp}. This completes the proof.
\qed

\begin{thm}\label{theorem}
For any positive integer $m$, there exist integer $N$, real numbers $a_k$ and $C_1, C_2>0$ such that for $n\geq N$
\begin{align}\label{p-p-b}
\sum_{k=0}^ma_k\wmu(n)^{-k}-C_1\wmu(n)^{-m-1}<\frac{\p(n+1)}{\p(n)}<\sum_{k=0}^ma_k\wmu(n)^{-k}+C_2\wmu(n)^{-m-1}.
\end{align}
\end{thm}

\proof
By \eqref{p-T-R} and Theorem \ref{wyn-bound}, for any $m\geq2$ there exists the integer $N_1(m)$ such that
\begin{equation*}
  |\p(n)/\wT(n)-1|<\left(\frac 32\right)^{m+1}\wmu(n)^{-m},~~\forall n\geq N_1(m).
\end{equation*}
We have
\begin{equation*}
  \wT(n)\left(1-\left(\frac 32\right)^{m+1}\wmu(n)^{-m}\right)<\p(n)<\wT(n)\left(1+\left(\frac 32\right)^{m+1}\wmu(n)^{-m}\right).
\end{equation*}
Since $\wmu(n)$ is a increasing function of $n$, we derive that
\begin{equation*}
  \frac {\wT(n+1)}{\wT(n)} \frac {1-\left(\frac 32\right)^{m+1}\wmu(n)^{-m}}{1+\left(\frac 32\right)^{m+1}\wmu(n)^{-m}}<\frac {\p(n+1)}{\p(n)}<\frac {\wT(n+1)}{\wT(n)} \frac {1+\left(\frac 32\right)^{m+1}\wmu(n)^{-m}}{1-\left(\frac 32\right)^{m+1}\wmu(n)^{-m}}.
\end{equation*}
We find that for $0<\lambda<1/3$,
\begin{equation*}
  \frac {1+\lambda}{1-\lambda}<1+3\lambda~~~ and~~~\frac {1-\lambda}{1+\lambda}>1-2\lambda.
\end{equation*}
According to
\[0<\frac {\left(\frac 32\right)^{m+1}}{\wmu(n)^m}<1/3,\]
we have that for all $n\geq N_1(m)$
\begin{equation}\label{p-p-bound}
  \frac {\wT(n+1)}{\wT(n)}\left(1-4\cdot2^m\wmu(n)^{-m}\right)<\frac {\p(n+1)}{\p(n)}<\frac {\wT(n+1)}{\wT(n)}\left(1+6\cdot2^m\wmu(n)^{-m}\right).
\end{equation}
By Lemma \ref{T-T-bound}, we can see that $\wT(n+1)/\wT(n)$ is bounded by a pair of polynomials in $\wmu(n)^{-1}$ whose difference is a polynomials in $\wmu(n)^{-1}$ with degree at least $m+1$. Combining \eqref{p-p-bound} and $\lim_{n\rightarrow+\infty}\wmu(n)=+\infty$, we have \eqref{p-p-b}. This completes the proof.
\qed
For any positive integer $m$, we further show that we can give the explicit numbers of these parameters. In this paper, we follow the Mathematica package of Hou and Zhang \cite{Hou-Zhang-2019}  to compute these parameters.  For example, we compute that for $n>66$
\begin{equation*}
 \sum_{k=0}^4a_k\wmu(n)^{-k}-\frac {160}{\wmu(n)^5}<\frac{\p(n+1)}{\p(n)}<\sum_{k=0}^4a_k\wmu(n)^{-k}+\frac {873}{\wmu(n)^5},
\end{equation*}
where
\begin{equation*}
 \sum_{k=0}^4a_k\wmu(n)^{-k}=1 +\frac {\pi^2}{2\wmu(n)} +\frac {-\pi^2 +\frac {\pi^4}8}{\wmu(n)^2} +\frac {\frac {\pi^2}2 - \frac {5 \pi^4}8 +\frac {\pi^6}{48}}{\wmu(n)^3} +\frac  {\frac {\pi^2}2 +\frac {5 \pi^4}4 -\frac {3 \pi^6}{16} +\frac {\pi^8}{384}}{\wmu(n)^4}.
\end{equation*}

This depends on an algorithm, so we give a specific example to present the calculating process in Section \ref{sec2}.

The following lemma given by Hou and Zhang \cite{Hou-Zhang-2019+} plays an important role in the proof of Theorem \ref{p/n-r-log}.
\begin{lem}\label{sn-l}
Let $\{a_n\}_{n\geq0}$ be a positive sequence such that $\mathscr{R}^2a_n=a_na_{n+2}/a_{n+1}^2$ has the following asymptotic expression
\begin{equation*}
  \mathscr{R}^2a_n=1+\frac c{n^\alpha}+\cdots+o\left(\frac c{n^\beta}\right),~~n\rightarrow \infty,
\end{equation*}
where $0<\alpha\leq\beta$. If $c<0$ and $\alpha <2$, then $\{a_n\}_{n\geq0}$ is asymptotically $[\beta /\alpha]$-log-concave.
\end{lem}

Now we are in a position to prove Theorem \ref{p/n-r-log}.

{\noindent\it Proof of Theorem \ref{p/n-r-log}.}
Based on \eqref{p-p-b}, we consider the bound of $\wmu(n+1)^{-r}$ and $1/h(\wmu(n)^{-1})$, where $h(\wmu(n)^{-1})$ is a polynomial in $\wmu(n)^{-1}$ with constant term $1$.

Let $c$ denote the tail term of $h(\wmu(n)^{-1})-1$. If $c>0$, there exists a positive integer $N$ such that
\begin{equation*}
  h(\wmu(n)^{-1})-1>0,~~\forall n\geq N.
\end{equation*}
By Taylor's Theorem, we have
\begin{align*}
  1-\lambda+\lambda^2-\cdots +(-1)^m\lambda^m-\lambda^{m+1}<\frac 1{h(\wmu(n)^{-1})}&=\frac 1{1+\left(h(\wmu(n)^{-1})-1\right)}\\[3pt]
  &<1-\lambda+\lambda^2-\cdots  +(-1)^m\lambda^m+\lambda^{m+1},
\end{align*}
where $\lambda=h(\wmu(n)^{-1})-1$. If $c<0$, there exists a positive integer $N$ such that
\begin{equation*}
  0<1-h(\wmu(n)^{-1})<\frac 12,~~\forall n\geq N.
\end{equation*}
So
\begin{align*}
  1+\lambda+\lambda^2+\cdots +\lambda^m+\lambda^{m+1}<\frac 1{h(\wmu(n)^{-1})}&=\frac 1{1-\left(1-h(\wmu(n)^{-1})\right)}\\[3pt]
  &<1+\lambda+\lambda^2+\cdots+\lambda^m+2\lambda^{m+1},
\end{align*}
where $\lambda=1-h(\wmu(n)^{-1})$.

We now consider the bound of $\wmu(n+1)^{-r}$. It is easy to get that
\begin{equation*}
\left(\frac {\wmu(n+1)}{\wmu(n)}\right)^{-r}=\left(1+\frac {\pi^2}{\wmu(n)^2}\right)^{-r/2}.
\end{equation*}
Then we can derive the bounds of $\left(\frac {\wmu(n+1)}{\wmu(n)}\right)^{-r}$ in a way similar to the estimation of $\frac {\wmu(n+1)}{\wmu(n)}$, thus get an estimation of $\wmu(n+1)^{-r}$.

Based on the above estimations, we compute the asymptotic expression of $\u_n:=\dfrac{\p(n-1)\p(n+1)}{\p(n)^2}$ by Mathematica, for any positive integer $m$,
\begin{equation*}
 \u_n=1-\frac {\pi}{4n^{3/2}}+\cdots+o\left(\frac 1{n^m}\right).
\end{equation*}
By Lemma \ref{sn-l}, we complete the proof.
\qed

\section{Proof of Theorems \ref{thm1}, \ref{thm2}, \ref{thm3} and Corollary \ref{cor1}-\ref{cor6}}\label{sec2}
In Section \ref{sec1}, we prove the asymptotic $r$-$\log$-concavity for the overpartition function. In this section, we study the $2$-$\log$-concavity as an example, stated in Corollary \ref{cor3}.
It's worth noting that we can derive $3$-$\log$-concavity (or others) in the same way.
But it could be more difficult with $r$ becoming larger. We also prove the conclusions what Theorem \ref{p/n-r-log} brings.

The key idea behind the proof of Theorem \ref{thm1} is lying in a detail analysis of the Theorem \ref{theorem} by exploiting the Equations \eqref{T-ratio} and \eqref{p-p-bound}.
More specifically, we shall proceed for a detail inquiry of the exact asymptotics for each of the factor present in $\wT(n+1)/\wT(n)$ explicitly by studying the Taylor expansion of the form $\sum_{m \geq 0}a_m(\sqrt{n})^{-m}$ upto order $7$ and bounding the error term.
This will set up a stage for the proof of Theorem \ref{thm2} and \ref{thm3}.

{\noindent\it Proof of Theorem \ref{thm1}.}
We recall the Equation \eqref{T-ratio}:
\begin{equation*}
\frac{\wT(n+1)}{\wT(n)}=\frac{\wmu(n+1)-1}{\wmu(n)-1}\cdot\frac{\wmu(n)^3}{\wmu(n+1)^3}\cdot e^{\wmu(n+1)-\wmu(n)}.
\end{equation*}
By Taylor's theorem, we get
\begin{equation*}
\frac{\wmu(n+1)-1}{\wmu(n)-1}=s^{(1)}_+(n)+O\Bigl(\dfrac{1}{n^4}\Bigr),
\end{equation*}
where
\begin{equation*}
\begin{split}
s^{(1)}_+(n)= 1+\dfrac{1}{2n}+\dfrac{1}{2\pi n^{3/2}}+\Bigl(\dfrac{1}{2\pi^2}-\dfrac{1}{8}\Bigr)\dfrac{1}{n^2}+\Bigl(\dfrac{1}{2\pi^3}-\dfrac{1}{8\pi}\Bigr)\dfrac{1}{n^{5/2}}&+\Bigl(\dfrac{1}{16}+\dfrac{1}{2\pi^4}-\dfrac{1}{8\pi^2}\Bigr)\dfrac{1}{n^3}\\
&+ \Bigl(\dfrac{1}{2\pi^5}-\dfrac{1}{8\pi^3}+\dfrac{1}{16\pi}\Bigr)\dfrac{1}{n^{7/2}}.
\end{split}
\end{equation*}
It is easy to observe that for $n \geq 1$,
\begin{equation}\label{thm1eqn1}
s^{(1)}_+(n)-\dfrac{2}{n^4}<\frac{\wmu(n+1)-1}{\wmu(n)-1}<s^{(1)}_+(n)+\dfrac{2}{n^4}.
\end{equation}
Similarly, for $n \geq 1$, we obtain
\begin{equation}\label{thm1eqn2}
s^{(2)}_+(n)-\dfrac{3}{n^4}<\frac{\wmu(n)^3}{\wmu(n+1)^3}<s^{(2)}_+(n)+\dfrac{3}{n^4},
\end{equation}
where
$$s^{(2)}_+(n)=1-\dfrac{3}{2n}+\dfrac{15}{8n^2}-\dfrac{35}{16n^3}.$$
For the factor $e^{\wmu(n+1)-\wmu(n)}$, we first estimate $\wmu(n+1)-\wmu(n)$ as follows; for $n \geq 1$,
\begin{equation}\label{thm1eqn3}
s^{(3,0)}_+(n)<\wmu(n+1)-\wmu(n) <s^{(3,0)}_+(n)+\dfrac{3}{n^4},
\end{equation}
where $$s^{(3,0)}_+(n)=\dfrac{\pi}{2\sqrt{n}}-\dfrac{\pi}{8n^{3/2}}+\dfrac{\pi}{16n^{5/2}}-\dfrac{5\pi}{128n^{7/2}}.$$
Now, expanding $e^{s^{(3,0)}_+(n)}$ and truncate the infinite series at the order $\dfrac{1}{n^{7/2}}$ and bound the error term that states for $n \geq 2$,
\begin{equation}\label{thm1eqn4}
\ e^{3/n^4}<1+\dfrac{4}{n^4}\ \ \text{and}\ \ s^{(3)}_+(n)-\dfrac{1}{n^4}<e^{s^{(3,0)}_+(n)} <s^{(3)}_+(n)+\dfrac{1}{n^4},
\end{equation}
where
$$s^{(3)}_+(n)=\sum_{m=0}^{7}s^{(3)}_{+,m}\Bigl(\dfrac{1}{\sqrt{n}}\Bigr)^m$$
with
\begin{align*}
&s^{(3)}_{+,0}=1, \quad
s^{(3)}_{+,1}=\dfrac{\pi}{2},  \quad
s^{(3)}_{+,2}=\dfrac{\pi^2}{8},  \quad
s^{(3)}_{+,3}=\dfrac{\pi (\pi^2-6)}{48},
\quad
s^{(3)}_{+,4}=\dfrac{\pi^2 (\pi^2-24)}{384},
\\[3pt]
&s^{(3)}_{+,5}=\dfrac{\pi (\pi^4-60\pi^2+240)}{3840},
\quad
s^{(3)}_{+,6}=\dfrac{\pi^2 (\pi^4-120\pi^2+1800)}{46080},
\\[3pt]
&s^{(3)}_{+,7}=\dfrac{\pi (\pi^6-210\pi^4+7560\pi^2-25200)}{645120}.
\end{align*}
From \eqref{thm1eqn3} and \eqref{thm1eqn4}, we get
\begin{equation}\label{thm1eqn5}
\Bigl(s^{(3)}_+(n)-\dfrac{1}{n^4}\Bigr)<e^{\wmu(n+1)-\wmu(n)}<\Bigl(s^{(3)}_+(n)+\dfrac{1}{n^4}\Bigr)\Bigl(1+\dfrac{4}{n^4}\Bigr).
\end{equation}
It can be easily checked that for $n \geq 1$,
\begin{equation}\label{thm1eqn6}
1-\dfrac{4\cdot2^8} {\wmu(n)^{8}}>1-\dfrac{1}{n^4}\ \ \text{and}\ \ 1+\dfrac{6\cdot2^8} {\wmu(n)^{8}}<1+\dfrac{1}{n^4}.
\end{equation}
Hence, by \eqref{thm1eqn1}, \eqref{thm1eqn2}, \eqref{thm1eqn5}, \eqref{thm1eqn6} and using \eqref{p-p-bound} with $m=8$, we obtain for all $n \geq 184$,
\begin{equation}\label{thm1eqn7}
L_{+}(n)<\dfrac{\p(n+1)}{\p(n)}<U_{+}(n),
\end{equation}
where
\begin{align*}
U_{+}(n)&= \Bigl(s^{(1)}_+(n)+\dfrac{2}{n^4}\Bigr)\Bigl(s^{(2)}_+(n)
+\dfrac{3}{n^4}\Bigr)\Bigl(s^{(3)}_+(n)+\dfrac{1}{n^4}\Bigr)
\Bigl(1+\dfrac{4}{n^4}\Bigr)\Bigl(1+\dfrac{1}{n^4}\Bigr),
\\[3pt]
L_{+}(n)&=\Bigl(s^{(1)}_+(n)-\dfrac{2}{n^4}\Bigr)
\Bigl(s^{(2)}_+(n)-\dfrac{3}{n^4}\Bigr)\Bigl(s^{(3)}_+(n)
-\dfrac{1}{n^4}\Bigr)\Bigl(1-\dfrac{1}{n^4}\Bigr).
\end{align*}
In the similar way stated before, we obtain for $n \geq 184$,
\begin{equation}\label{thm1eqn8}
L_{-}(n)<\dfrac{\p(n-1)}{\p(n)}<U_{-}(n),
\end{equation}
where
\begin{align*}
U_{-}(n)&= \Bigl(s^{(1)}_-(n)+\dfrac{1}{n^4}\Bigr)\Bigl(s^{(2)}_-(n)+\dfrac{3}{n^4}\Bigr)
\Bigl(s^{(3)}_-(n)+\dfrac{1}{n^4}\Bigr)\Bigl(1+\dfrac{1}{n^4}\Bigr),\\ L_{-}(n)&=\Bigl(s^{(1)}_-(n)-\dfrac{1}{n^4}\Bigr)\Bigl(s^{(2)}_-(n)\Bigr)\Bigl(s^{(3)}_-(n)\Bigr)\Bigl(1-\dfrac{1}{n^4}\Bigr)^2,
\end{align*}
with
\begin{align*}
s^{(1)}_-(n)&=1-\dfrac{1}{2n}-\dfrac{1}{2\pi n^{3/2}}-\Bigl(\dfrac{1}{2\pi^2}+\dfrac{1}{8}\Bigr)\dfrac{1}{n^2}
-\Bigl(\dfrac{1}{2\pi^3}+\dfrac{1}{8\pi}\Bigr)\dfrac{1}{n^{5/2}}
\\[3pt]
&\quad-\Bigl(\dfrac{1}{16}+\dfrac{1}{2\pi^4}+\dfrac{1}{8\pi^2}\Bigr)\dfrac{1}{n^3}
-\Bigl(\dfrac{1}{2\pi^5}+\dfrac{1}{8\pi^3}+\dfrac{1}{16\pi}\Bigr)\dfrac{1}{n^{7/2}},
\\[3pt]
s^{(2)}_-(n)&=1+\dfrac{3}{2n}+\dfrac{15}{8n^2}+\dfrac{35}{16n^3},
\\[3pt]
s^{(3)}_-(n)&=\sum_{m=0}^{7}s^{(3)}_{-,m}\Bigl(\dfrac{1}{\sqrt{n}}\Bigr)^m,
\end{align*}
along with
\begin{align*}
&s^{(3)}_{-,0}=1,\quad
s^{(3)}_{-,1}=-\dfrac{\pi}{2},\quad
s^{(3)}_{-,2}=\dfrac{\pi^2}{8},\quad
s^{(3)}_{-,3}=-\dfrac{\pi (\pi^2+6)}{48},\quad
s^{(3)}_{-,4}=\dfrac{\pi^2 (\pi^2+24)}{384},
\\[3pt]
&s^{(3)}_{-,5}=-\dfrac{\pi (\pi^4+60\pi^2+240)}{3840}, \quad
s^{(3)}_{-,6}=\dfrac{\pi^2 (\pi^4+120\pi^2+1800)}{46080},
\\[3pt]
&s^{(3)}_{-,7}=-\dfrac{\pi (\pi^6+210\pi^4+7560\pi^2+25200)}{645120}.
\end{align*}
Now by \eqref{thm1eqn7} and \eqref{thm1eqn8}, it follows that for $n \geq 184$,
\begin{equation}\label{thm1eqn9}
L_{+}(n)\cdot L_{-}(n)<\u(n)<U_{+}(n)\cdot U_{-}(n).
\end{equation}
It can be readily checked that for $n \geq 2$,
\begin{equation*}
U_{+}(n)\cdot U_{-}(n) <s(n)+\dfrac{20}{n^4}\ \ \text{and}\ \ L_{+}(n)\cdot L_{-}(n) >s(n)-\dfrac{15}{n^4}.
\end{equation*}
This finishes the proof of \eqref{eqn1} for $n \geq 184$. For the rest $37 \leq n \leq 183$, one can check numerically in Mathematica.
\qed

{\noindent\it Proof of Corollary \ref{cor1}.}
It is easy to check that $s_n+\frac{20}{n^4}<1$ for $n \geq 5$ and therefore from \eqref{eqn1}, can conclude that $\u_n<1$ for $n \geq 37$ which is equivalent to say that $\{\p(n)\}_{n \geq 37}$ is $\log$-concave. For $4 \leq n \leq 37$, we did numerical checking in Mathematica.
\qed

{\noindent\it Proof of Corollary \ref{cor2}.}
Note that for $n \geq 5$
\[\Bigl(s_n-\frac{15}{n^4}\Bigr)\Bigl(1+\dfrac{\pi}{4n^{3/2}}\Bigr)>1\]
and therefore from \eqref{eqn1}, it follows that for $n \geq 37$ \[\u_n\Bigl(1+\dfrac{\pi}{4n^{3/2}}\Bigr)>1\]
which is equivalent to \eqref{eqn2}. For $2 \leq n \leq 37$, we did numerical checking in Mathematica.
\qed

Define
\begin{equation}\label{def}
U_n:= s_n+\dfrac{20}{n^4}\ \ \text{and}\ \ L_n:=s_n-\dfrac{15}{n^4}.
\end{equation}

{\noindent\it Proof of Theorem \ref{thm2}.} Using \eqref{eqn1} from Theorem \ref{thm1}, it follows that for $n \geq 37$,
\begin{equation}\label{thm2eqn1}
\dfrac{(1-U_n)^2}{U^2_n(1-L_{n-1})(1-L_{n+1})}<\dfrac{(1-\u_n)^2}{\u^2_n(1-\u_{n-1})(1-\u_{n+1})} < \dfrac{(1-L_n)^2}{L^2_n(1-U_{n-1})(1-U_{n+1})}.
\end{equation}
Moreover, it can be readily checked that for $n \geq 29$,
\begin{equation*}
\begin{split}
\dfrac{(1-L_n)^2}{L^2_n(1-U_{n-1})(1-U_{n+1})} &< t_n+\dfrac{120}{n^{5/2}},
\\[3pt]
\dfrac{(1-U_n)^2}{U^2_n(1-L_{n-1})(1-L_{n+1})}&> t_n-\dfrac{120}{n^{5/2}}.
\end{split}
\end{equation*}
We conclude the proof of \eqref{eqn3} by checking numerically for $27 \leq n \leq 36$ in Mathematica.
\qed

{\noindent\it Proof of Corollary \ref{cor3}.}
It is equivalent to show that for $n \geq 42$ \[\dfrac{(1-\u_n)^2}{\u^2_n(1-\u_{n-1})(1-\u_{n+1})}>1.\]
From the fact that for all $n \geq 99$,
\[t_n-\dfrac{120}{n^{5/2}}>1\]
and by \eqref{eqn3}, the proof is finished after the numerical verification for $42 \leq n \leq 98$ in Mathematica.
\qed

{\noindent\it Proof of Corollary \ref{cor4}.}
Since for all $n \geq 1176$,
\[t_n+\dfrac{120}{n^{5/2}}<1+\dfrac{\pi}{2n^{3/2}}.\]
By \eqref{eqn3}, we conclude the proof. For the rest, we can check numerically \eqref{eqn4} for $52 \leq n \leq 1175$ in Mathematica.
\qed

{\noindent\it Proof of Theorem \ref{thm3}.} Following the definition given before and by \eqref{eqn1} of Theorem \ref{thm1}, it follows that for $n \geq 37$,
\begin{equation}\label{thm3eqn1}
\dfrac{4(1-U_n)(1-U_{n+1})}{(1-L_{n}L_{n+1})^2}< \dfrac{4(1-\u_n)(1-\u_{n+1})}{(1-\u_n\u_{n+1})^2}< \dfrac{4(1-L_n)(1-L_{n+1})}{(1-U_{n}U_{n+1})^2}.
\end{equation}
It is easy to observe that for $n \geq 99$,
\begin{equation*}
\begin{split}
\dfrac{4(1-L_n)(1-L_{n+1})}{(1-U_{n}U_{n+1})^2} &< v_n+\dfrac{101}{n^{5/2}},
\\[3pt]
\dfrac{4(1-U_n)(1-U_{n+1})}{(1-L_{n}L_{n+1})^2}&> v_n-\dfrac{120}{n^{5/2}}.
\end{split}
\end{equation*}
We conclude the proof of \eqref{eqn5} by checking numerically for $2 \leq n \leq 98$ in Mathematica.
\qed

{\noindent\it Proof of Corollary \ref{cor5}.}
We observe that $v_n-\dfrac{120}{n^{5/2}}>1$ for all $n \geq 180$ and hence by \eqref{eqn5}, it follows immediately that $\{\p(n)\}_{n \geq 180}$ satisfies higher order Tur\'{a}n inequality and for $16 \leq n \leq 179$, we verified numerically in Mathematica.
\qed

{\noindent\it Proof of Corollary \ref{cor6}.}
	It is straightforward to check that $v_n+\dfrac{101}{n^{5/2}}< 1+\dfrac{\pi}{4n^{3/2}}$ for all $n \geq 4179$ in \eqref{eqn5}. To finish the proof of \eqref{eqn6}, it remains to verify for $2 \leq n \leq 4178$, which was done by numerical verification in Mathematica.
\qed

\vspace{0.5cm}
 \baselineskip 15pt
{\noindent\bf\large{\ Acknowledgements}} \vspace{7pt} \par
The first author would like to acknowledge that the research was funded by the Austrian Science Fund (FWF): W1214-N15, project DK13.
The second author would like to acknowledge that the research was supported by the National Natural Science Foundation of China (Grant Nos. 12001182 and 12171487),  the Fundamental Research Funds for the Central Universities (Grant No. 531118010411) and Hunan Provincial Natural Science Foundation of China (Grant No. 2021JJ40037).

\end{document}